\documentclass[reqno,12pt]{amsart}
\usepackage[margin=30mm]{geometry}
\usepackage{amssymb,amsmath,color}
\usepackage{hyperref}
\usepackage{scalerel}
\usepackage{soul}

\newtheorem{theorem}{Theorem}[section]

\theoremstyle{definition}

\newtheorem{remark}{Remark}

\def\del{\partial}
\newcommand{\ds}{\displaystyle}

\definecolor{red}{rgb}{.8,.1,.1}
\definecolor{blue}{rgb}{.2,.2,.8}
\definecolor{teal}{rgb}{0.1, .502, .502}

\begin{document}
\title[Franklin and Beck-type identities for $n$-color partitions]{Franklin's identity for $n$-color partitions\\ and companion Beck-type identities}

\author[C. Ballantine]{Cristina Ballantine}

\address{Department of Mathematics and Computer Science\\ 
College of the Holy Cross \\ Worcester, MA 01610, USA} 
\email{cballant@holycross.edu} 

\author[R. Tauraso]{Roberto Tauraso}

\address{Dipartimento di Matematica, 
Università di Roma ``Tor Vergata'', 00133 Roma, Italy}
\email{tauraso@mat.uniroma2.it}
\urladdr{https://robertotauraso.github.io/}




\begin{abstract} We show that some classical identities valid for ordinary partitions have precise analogues for $n$-color partitions, that is partitions in which a part of size $n\geq 1$ can occur in colors $1, 2, \ldots, n$.
For $r \ge 2$ and $j \ge 0$, we write $\mathcal{O}_{j,r}(m)$ and $\mathcal{D}_{j,r}(m)$ for the sets of $n$-color partitions of $m$ with, respectively, exactly $j$ different parts whose size and color are divisible by $r$, and exactly $j$ different parts occurring at least $r$ times. We prove an $n$-color version of Franklin's theorem, $|\mathcal{O}_{j,r}(m)| = |\mathcal{D}_{j,r}(m)|$, along with two Beck-type identities. We give both analytic and combinatorial proofs for all theorems.  
\end{abstract}

\maketitle

\section{Introduction}

A \emph{partition} $\lambda = (\lambda_1, \lambda_2, \dots, \lambda_r)$ of a positive integer $m$ is a finite sequence of non-increasing positive integers $\lambda_i$, called \emph{parts}, such that $m = \sum_{i=1}^r \lambda_i$. The partition function $p(m)$ counts the partitions of $m$, with generating function (\cite{Andrews98} and \cite[A000041]{OEIS}) 
\begin{align}\label{pgf}
\sum_{m=0}^{\infty} p(m)q^m = \prod_{m=1}^{\infty} \dfrac{1}{1-q^m}
=1+q+2q^2+3q^3+5q^4+7q^5+ 11q^6+\dots
\end{align}

In 2017, Beck proposed two conjectural identities that relate the number of parts in partitions into odd parts, the number of parts in partitions into distinct parts, and the number of partitions in which even parts or a repeated part play a special role. Andrews \cite{Andrews17} proved these conjectures analytically and later Yang \cite{Yang19}
provided combinatorial proofs.

\begin{enumerate}

\item[i)] The number of partitions of $m$ in which the set of even parts has only one element is equal to the difference between the number of parts in the odd partitions of $m$ and the number of parts in the distinct partitions of $m$ (\cite[A090867]{OEIS}).

\item[ii)] The number of partitions of $m$  with exactly one part that occurs three times, while all other parts occur only once is equal to the difference between the number of parts in the distinct partitions of $m$ and the number of distinct parts in the odd partitions of $m$ (\cite[A265251]{OEIS}).

\end{enumerate}

Since then, several authors have obtained refinements and generalizations of Beck's identities. Among the relevant literature, we mention \cite{FuTang17,Yang19,BW21, BW23}.

A now-classic, though perhaps not widely known, variation of integer partitions is that of $n$-color partitions (also referred to as partitions with ``$n$ copies of $n$").
An $n$-color partition of $m$ is a partition in which each part of size $n$ can appear in $n$ distinct colors, denoted by a subscript $i \in \{1, 2, \dots, n\}$. For a given colored part $n_i$, we refer to $n$ as the size of the part and to $i$ as the subscript (or color) of the part. For example, the six $n$-color partitions of $3$ are $(1_1,1_1,1_1)$, $(2_1,1_1)$, $(2_2,1_1)$, $(3_1)$, $(3_2)$, $(3_3)$.
This notion was introduced by Agarwal and Andrews \cite{Agarwal-Andrews87} in connection with Rogers--Ramanujan-type identities. Writing $p_{nc}(m)$ for the number of $n$-color partitions of $m$, the generating  function of $p_{nc}(m)$ (\cite[A000219]{OEIS}) is given by
\begin{align*}
\sum_{m=0}^{\infty} p_{nc}(m)q^m = \prod_{m=1}^{\infty} \dfrac{1}{(1-q^m)^m}
=1+q+3q^2+6q^3+13q^4+24q^5+ 48q^6+\dots
\end{align*}
Notice that this generating function coincides with that of the number of plane partitions. In \cite{Agarwal02}, Agarwal provides an indirect bijection proving,  through some intermediate combinatorial objects, that  that plane partitions and $n$-color partitions of $m$ are equinumerous.

Recently, several authors have shown that classical identities for ordinary partitions admit analogues in the $n$-color setting (see, for example, \cite{Bandyopadhyay-Baruah21, Sharma-Kaur25, NR26}).
Inspired by the work of Ballantine and Welch \cite{BW21, BW23}, we establish $n$-color analogues of Beck's original identities and several of their generalizations. More precisely, for $r \ge 2$ and $j \ge 0$, let $\mathcal{O}_{j,r}(m)$ denote the set of $n$-color partitions of $m$ with exactly $j$ different parts whose size and subscript  are both divisible by $r$, and let $\mathcal{D}_{j,r}(m)$ denote the set of $n$-color partitions of $m$ with exactly $j$ different parts occurring at least $r$ times. The results of this paper are as follows.
\begin{itemize}
    
\item In Section~2 we obtain an $n$-color version of Franklin's theorem (Theorem \ref{F}), proving that $|\mathcal{O}_{j,r}(m)| = |\mathcal{D}_{j,r}(m)|$ for all $r \ge 2$ and $m, j \ge 0$. We give both a generating-function proof, via trivariate generating functions recording the number of parts, and an explicit combinatorial bijection $\varphi_j : \mathcal{O}_{j,r}(m) \to \mathcal{D}_{j,r}(m)$ built from a Glaisher-type transformation.

\item In Section~3 we prove an $n$-color analogue of a Beck-type identity (Theorem~\ref{B1}), relating the counting sequence $(j+1)|\mathcal{O}_{j+1,r}(m)| - j|\mathcal{O}_{j,r}(m)|$ to the difference in the total number of parts in all partitions in $\mathcal{O}_{j,r}(m)$ and the total number of parts in all partitions in $\mathcal{D}_{j,r}(m)$. Again, both an analytic proof, through logarithmic differentiation of the generating functions of Section~2, and a combinatorial proof are given.

\item Section~4 presents a second $n$-color Beck-type identity (Theorem~\ref{B2}), expressing the difference $T_{j+1,r}(m) - T_{j,r}(m)$, where $T_{j,r}(m)$ is the number of parts with multiplicity strictly between $r$ and $2r$ in all partitions of $\mathcal{D}_{j,r}(m)$, in terms of the difference in the number of \emph{distinct} parts between $\mathcal{D}_{j,r}(m)$ and $\mathcal{O}_{j,r}(m)$.
\end{itemize}

Together, these results show that the Beck-type phenomenon first observed for ordinary partitions extends to the family of $n$-color partitions.

\section{$n$-color version of Franklin's Theorem}

\begin{theorem} \label{F} For $r\geq2$ and $m,j\geq 0$, we have
\begin{equation}\label{FTid}
|\mathcal{O}_{j,r}(m)|=|\mathcal{D}_{j,r}(m)|.
\end{equation} 
\end{theorem}

\begin{proof}[Analytic Proof] Let $r\geq 2$. For $k\geq 0$, denote by $\mathcal{O}_{j,r}(k,m)$, respectively $\mathcal{D}_{j,r}(k,m)$,  the subset of partitions in $\mathcal{O}_{j,r}(m)$, respectively $\mathcal{D}_{j,r}(m)$, with $k$ parts. We start with the trivariate generating functions for the sequences $\{|\mathcal{O}_{j,r}(k,m)|\}$ and $\{|\mathcal{D}_{j,r}(k,m)|\}$: 
$$\mathcal{O}_r(z,w,q):=
\sum_{m=0}^\infty\sum_{k=0}^\infty\sum_{j=0}^\infty|\mathcal{O}_{j,r}(k,m)|z^kw^jq^m,$$ 
and 
$$\mathcal{D}_r(z,w,q):=\sum_{m=0}^\infty\sum_{k=0}^\infty\sum_{j=0}^\infty|\mathcal{D}_{j,r}(k,m)|z^kw^jq^m.$$ 
We have  
\begin{align*}
\mathcal{O}_r(z,w,q)
&=\prod_{m=1}^\infty(1+w\sum_{i=1}^{\infty}(zq^{rm})^i)^m\cdot
\prod_{m=1}^\infty\frac{1}{(1-zq^{rm})^{(r-1)m}}\cdot\prod_{\substack{m=1\\ r\nmid m}}^\infty\frac{1}{(1-zq^m)^m}   \\  
&=\prod_{m=1}^\infty\left(1+\frac{wzq^{rm}}{1-zq^{rm}}\right)^m\cdot \prod_{m=1}^\infty\left(\frac{1-zq^{rm}}{1-zq^m}\right)^m\\
&=\prod_{m=1}^{\infty}\left(\frac{1-(1-w)zq^{rm}}{1-zq^{m}}\right)^m
\end{align*} 
and  
\begin{align*}
\mathcal{D}_r(z,w,q)
&=\prod_{m=1}^\infty\Big(\sum_{i=0}^{r-1}(zq^{m})^i+w\sum_{i=r}^{\infty}(zq^{m})^i\Big)^m\\
&=\prod_{m=1}^\infty\Big(\frac{1-z^rq^{rm}}{1-zq^m}+\frac{wz^rq^{rm}}{1-zq^m}\Big)^m\\
&=\prod_{m=1}^{\infty}\left(\frac{1-(1-w)z^rq^{rm}}{1-zq^{m}}\right)^m.
\end{align*} 
Hence, we find that
$$\sum_{m=0}^\infty\sum_{j=0}^\infty|\mathcal{O}_{j,r}(m)|w^jq^m
=\mathcal{O}_r(1,w,q)=\prod_{m=1}^{\infty}\left(\frac{1-(1-w)q^{rm}}{1-q^{m}}\right)^m$$
is equal to
$$\sum_{m=0}^\infty\sum_{j=0}^\infty|\mathcal{D}_{j,r}(m)|w^jq^m
=\mathcal{D}_r(1,w,q)
=\prod_{m=1}^{\infty}\left(\frac{1-(1-w)q^{rm}}{1-q^{m}}\right)^m.$$
\end{proof}
\begin{proof}[Combinatorial Proof]

We first establish a bijection $$\varphi_0:\mathcal O_{0,r}(m)\to \mathcal D_{0,r}(m).$$ The bijection $\varphi_0$ is closely related to Glaisher's bijection used to prove that the number of partitions of $m$ with no part divisible by $r$ equals the number of partitions of $m$ with no part occurring more than $r-1$ times. Let $\lambda\in \mathcal O_{0,r}(m)$ be an $n$-color partition such that if $a_i$ is a part of $\lambda$ of size $a$ and subscript $i$, $1\leq i\leq a$, then $r\nmid \gcd(a,i)$. Denote by $m_\lambda(a_i)$ the multiplicity of $a_i$ in $\lambda$, that is, the number of times $a_i$ occurs as a part in $\lambda$. We write $m_\lambda(a_i)$ in base $r$, $$m_\lambda(a_i)=\sum_{j\geq 0}s_jr^j,$$ with $0\leq s_j<r$ for all $j\geq 0$. The partition $\varphi_0(\lambda)$ is obtained from $\lambda$ as follows:  for each part $a_i$ with $m_\lambda(a_i)\geq r$ we replace the multiset of all parts equal to $a_i$ by the multiset consisting of $s_j$ parts of size $ar^j$ and subscript $ir^j$ for each  $j\geq 0$. Then $\varphi_0(\lambda)\in  \mathcal D_{0,r}(m)$. For the inverse transformation, let $\mu\in  \mathcal D_{0,r}(m)$. The partition $\varphi_0^{-1}(\mu)$ is obtained from $\mu$ as follows: for each part $a_i$ in $\mu$ with $r\mid \gcd(a,i)$, let $t\geq 1$ be such that $r^t\mid \gcd(a,i)$ and $r^{t+1}\nmid \gcd(a,i)$. Each such part $a_i$ is replaced by $r^t$ parts of size $\displaystyle\frac{a}{r^t}$ and subscript $\displaystyle\frac{i}{r^t}$. Then $\varphi_0^{-1}(\mu)\in \mathcal O_{0,r}(m)$.\smallskip

Next, let $j\geq 1$. We define a bijection $$\varphi_j:\mathcal O_{j,r}(m)\to \mathcal D_{j,r}(m).$$  For the remainder of the article we identify a partition with its multiset of parts and use multiset operations on  multisets of parts of partitions. 

Let $\lambda\in \mathcal O_{j,r}(m)$. We write $\lambda=\lambda^{r\nmid}\cup \lambda^{r\mid}$,   where $\lambda^{r\nmid}\in \mathcal O_{0,r}(m-|\lambda^{r\mid}|)$ and $\lambda^{r\mid}$ is an $n$-color partition such that if $a_i$ is a part of $\lambda^{r\mid}$ then $r\mid \gcd(a,i)$. Notice that $\lambda^{r\mid}$ has $j$ different parts (possibly repeated). Let $\psi(\lambda^{r\mid})$ be the $n$-color partition obtained from $\lambda^{r\mid}$ as follows: each part $a_i$ in $\lambda^{r\mid}$ is replaced by $r$ parts of size $\displaystyle \frac{a}{r}$ and subscript $\displaystyle \frac{i}{r}$.
 We define 
 $$\varphi_j(\lambda):=\varphi_0(\lambda^{r\nmid}) \cup \psi( \lambda^{r\mid}) \in \mathcal D_{j,r}(m).$$ 
 For the inverse transformation, let $\mu\in \mathcal D_{j,r}(m)$. We write $\mu=\mu^{<r} \cup \mu^r$, where $\mu^{<r}\in \mathcal D_{0,r}(m-|\mu^r|)$ and $\mu^r$ is an $n$-color partition with all parts having multiplicity divisible by $r$. Then 
 $$\varphi_j^{-1}(\mu):=\varphi_0^{-1}(\mu^{<r})\cup \psi^{-1}(\mu^r),$$
 where $\psi^{-1}(\mu^r)$ is the $n$-color partition obtained from $\mu^r$ as follows: for each different part $a_i$ in $\mu^r$, we replace the multiset of all parts $a_i$ by the multiset consisting of $\ds\frac{m_{\mu^r}(a_i)}{r}$ parts all of size $ra$ and subscript $ri$.
\end{proof}

    \begin{remark}
         The case $j=0, r=2$ appears as \cite[Corollary 4]{NR26}. The combinatorial proof of \cite[Theorem 5]{NR26} is similar to the proof above. 
    \end{remark}

\section{$n$-color version of a Beck-type identity}

\begin{theorem} \label{B1}
For $r\geq2$ and $m,j\geq0$, we have
\begin{equation}\label{eqnBG}
(j+1)\left|\mathcal{O}_{j+1,r}(m)\right|-j\left|\mathcal{O}_{j,r}(m)\right|=\frac{1}{r-1}\cdot\Big(\sum_{\lambda\in\mathcal{O}_{j,r}(m)}\ell(\lambda)-\sum_{\lambda\in\mathcal{D}_{j,r}(m)}\ell(\lambda)\Big),
\end{equation}
 where $\ell(\lambda)$ is the number of parts of $\lambda$.
\end{theorem}

\begin{proof}[Analytic Proof] We begin by observing that
$$\sum_{m=0}^\infty\sum_{j=0}^\infty \Big(\sum_{\lambda\in\mathcal{O}_{j,r}(m)}\ell(\lambda)-\sum_{\lambda\in\mathcal{D}_{j,r}(m)}\ell(\lambda)\Big)w^jq^m  
= \left.\frac{\del}{\del z}\right|_{z=1}\!\!\!\!\!\!\bigl(\mathcal{O}_r(z,w,q)-\mathcal{D}_r(z,w,q)\bigr).$$ 
Using logarithmic differentiation, we find
\begin{align*}
\left.\frac{\del}{\del z}\right|_{z=1}\!\!\!\!\!\!\mathcal{O}_r(z,w,q)
&= \mathcal{O}_r(1,w,q)\sum_{m=1}^\infty m\left(
-\frac{(1-w)q^{rm}}{1-(1-w)q^{rm}}+\frac{q^{m}}{1-q^{m}}\right)
\end{align*}
and
\begin{align*}
\left.\frac{\del}{\del z}\right|_{z=1}\!\!\!\!\!\!\mathcal{D}_r(z,w,q)
&= \mathcal{D}_r(1,w,q)\sum_{m=1}^\infty m\left(
-\frac{r(1-w)q^{rm}}{1-(1-w)q^{rm}}+\frac{q^{m}}{1-q^{m}}\right).
\end{align*}
Thus, recalling that  $\mathcal{O}_r(1,w,q)=\mathcal{D}_r(1,w,q)$ by \eqref{FTid}, it follows  that
$$
\left.\frac{\del}{\del z}\right|_{z=1}\!\!\!\!\!\!
\bigl(\mathcal{O}_r(z,w,q)-\mathcal{D}_r(z,w,q)\bigr)=
(r-1)(1-w)\mathcal{O}_r(1,w,q)
\sum_{m=1}^\infty 
\frac{mq^{rm}}{1-(1-w)q^{rm}},$$
leading to
\begin{align}\label{LHSBG}
\frac{1}{r-1}\sum_{m=0}^\infty\sum_{j=0}^\infty \Big(\sum_{\lambda\in\mathcal{O}_{j,r}(m)}\ell(\lambda)&-\sum_{\lambda\in\mathcal{D}_{j,r}(m)}\ell(\lambda)\Big)w^jq^m 
\nonumber\\
&= (1-w)\mathcal{O}_r(1,w,q)
\sum_{m=1}^\infty 
\frac{mq^{rm}}{1-(1-w)q^{rm}}.
\end{align}
On the other hand, again by logarithmic differentiation,
\begin{align}\label{RHSBG}\sum_{m=0}^\infty\sum_{j=0}^\infty
\Big((j+1)\left|\mathcal{O}_{j+1,r}(m)\right|&-j\left|\mathcal{O}_{j,r}(m)\right|\Big)w^jq^m\nonumber\\
&=(1-w) \,
\frac{\del}{\del w}\bigl(\mathcal{O}_r(1,w,q)\bigr)\nonumber\\
&=(1-w)\,\mathcal{O}_r(1,w,q)\sum_{m=1}^{\infty}
\frac{mq^{rm}}{1-(1-w)q^{rm}}.
\end{align}
Comparing \eqref{LHSBG} with \eqref{RHSBG} reveals that both sides are identical, thus completing the proof.
\end{proof}

\begin{proof}[Combinatorial Proof] We first prove the case $j=0$. The proof is similar to that in \cite[Section 3]{Yang19}. 

For a non-negative integer $u$, if $u=\sum_{j\geq 0}a_j r^j$ with $0\leq a_j<r$ is the representation of $u$ in base $r$, let $p_r(u):=\sum_{j\geq 0}a_j$ be the sum of the digits of $u$ in base $r$. Then the bijection $\varphi_0$ defined in the combinatorial proof of Theorem \ref{F}, shows that 
\begin{align*} \sum_{\lambda\in\mathcal{O}_{0,r}(m)}\ell(\lambda)-\sum_{\lambda\in\mathcal{D}_{0,r}(m)}\ell(\lambda)& =\sum_{\lambda\in\mathcal{O}_{0,r}(m)}\Big(\ell(\lambda)-\varphi_0(\lambda)\Big)\\ & =\sum_{\lambda\in\mathcal{O}_{0,r}(m)}\sum_{a_i\in \lambda}\Big(m_\lambda(a_i)-p_r(m_\lambda(a_i))\Big). 
\end{align*}

As in \cite{Yang19}, for each $\lambda\in \mathcal O_{0,r}(m)$ and each $a_i\in \lambda$, we define a  subset $\mathcal O_{\lambda,r,a_i}(m)\subseteq \mathcal O_{1,r}(m)$ as follows. Let $M(a_i)=\lfloor \log_r(m_\lambda(a_i))\rfloor$. For $1\leq s\leq M(a_i)$, $1\leq t\leq \left\lfloor \frac{m_\lambda(a_i)}{r^s}\right\rfloor$, let $\lambda^{(a_i, s, t)}$ be the partition obtained from $\lambda$ by replacing $tr^s$ parts equal to $a_i$ by $t$ parts equal to $(ar^s)_{ir^s}$.  Clearly, $(ar^s)_{ir^s}$ is  the only part in $\lambda^{(a_i, s, t)}$ such that both the size and the subscript are divisible by $r$. Let 
$$\mathcal O_{\lambda,r, a_i}(m):=\Big\{\lambda^{(a_i, s, t)} \mid 1\leq s\leq M(a_i), 1\leq t\leq \left\lfloor \frac{m_\lambda(a_i)}{r^s}\right\rfloor\Big\}.$$ 
If $m_\lambda(a_i)<r$, then $\mathcal O_{\lambda,r,a_i}(m)=\emptyset$. 
Moreover, it follows from the definition that the sets $\mathcal O_{\lambda,r, a_i}(m)$ are disjoint as $\lambda\in \mathcal O_{0,r}(m)$ and $a_i\in \lambda$ with $m_\lambda(a_i)\geq r$ vary.
If $m_\lambda(a_i)\geq r$ and $(s,t)\neq (s',t')$, then $\lambda^{(a_i, s, t)}\neq \lambda^{(a_i, s', t')}$. 

The counting argument in the proof of \cite[Lemma 3.2]{Yang19} shows that $$|\mathcal O_{\lambda,r, a_i}(m)|=\frac{m_\lambda(a_i)-p_r(m_\lambda(a_i))}{r-1}.$$ It remains to prove that $$\mathcal O_{1,r}(m)\subseteq \bigcup_{\lambda\in \mathcal O_{0,r}(m)}\left(\bigcup_{a_i\in \lambda}\mathcal O_{\lambda,r, a_i}(m)\right).$$   Let $\mu \in \mathcal O_{1,r}(m)$ and suppose that $a_i$ is the only part in $\mu$ such that $r$ divides both the size and the subscript of the part. Let $v\geq 1$ be such that $r^v\mid \gcd(a,i)$ and $r^{v+1}\nmid \gcd(a,i)$. Let $\lambda$ be the partition obtained from $\mu$ by replacing each part equal to $a_i$ by $r^v$ parts of size $b:=\frac{a}{r^v}$ and subscript $k:=\frac{i}{r^v}$. Clearly, $\lambda \in \mathcal O_{0,r}(m)$, the pair $(\lambda,b_k)$ is uniquely determined by $\mu$, and   $\mu \in \mathcal O_{\lambda, r, b_k}(m)$. Therefore, 
$$|\mathcal O_{1,r}(m)|=\frac{1}{r-1}\Big(\sum_{\lambda\in\mathcal{O}_{0,r}(m)}\ell(\lambda)-\sum_{\lambda\in\mathcal{D}_{0,r}(m)}\ell(\lambda)\Big).$$
Next, let $j\geq 1$. The proof in this case is similar to that of \cite[Theorem 1.2]{BW23}. We set $$b_{j,r}(m):=\sum_{\lambda\in\mathcal{O}_{j,r}(m)}\ell(\lambda)-\sum_{\lambda\in\mathcal{D}_{j,r}(m)}\ell(\lambda).$$
With the notation in the combinatorial proof of Theorem \ref{F},  we obtain 
\begin{align*}b_{j,r}(m)& = \sum_{\lambda\in\mathcal{O}_{j,r}(m)} \Big( \ell(\lambda)-\ell(\varphi_j(\lambda))\Big)\\ 
& =  \sum_{\lambda\in\mathcal{O}_{j,r}(m)}\Big(\ell(\lambda^{r\nmid})-\ell(\varphi_0(\lambda^{r\nmid}))+ \ell(\lambda^{r\mid})-\ell(\psi(\lambda^{r\mid}))\Big).\end{align*} 
To rewrite the sum above, we define $\mathcal A_{j,r}(s)$ to be the set of $n$-color partitions $\beta$ of $s$ in which all parts have both the size and the subscript divisible by $r$ and such that $\beta$ has exactly $j$ different parts (possibly repeated). Then 
\begin{align*}b_{j,r}(m)& = \sum_{s\geq 0}\sum_{\beta\in \mathcal A_{j,r}(s)}\sum_{\alpha\in \mathcal O_{0,r}(m-s)}\Big(\ell(\alpha)-\ell(\varphi_0(\alpha))+\ell(\beta)-\ell(\psi(\beta)) \Big).\end{align*} 
Note that the set $\mathcal A_{j,r}(s)$ may be empty for certain values of $s$.
From the combinatorial proof of Theorem \ref{F}, it follows that
$$\ell(\beta)-\ell(\psi(\beta))=(1-r)\ell(\beta).$$ 
Moreover, from the combinatorial proof of case $j=0$ of this theorem, we have $$\sum_{\alpha\in \mathcal O_{0,r}(m-s)}\Big(\ell(\alpha)-\ell(\varphi_0(\alpha))\Big)=(r-1)|\mathcal O_{1,r}(m-s)|.$$ Thus $$b_{j,r}(m) = \sum_{s\geq 0}\sum_{\beta\in \mathcal A_{j,r}(s)}(r-1)\Big(|\mathcal O_{1,r}(m-s)|-\ell(\beta)|\mathcal O_{0,r}(m-s)|\Big). $$ If $\mathcal A_{j,r}(s)\neq \emptyset$, fix $\beta \in \mathcal A_{j,r}(s)$. We define a mapping $\zeta_\beta$ on the set of $n$-color partitions of $m-s$ by $$\zeta_\beta(\mu):=\mu \cup \beta.$$ 
Restricted to the set $\mathcal O_{0,r}(m-s)$, the mapping $\zeta_\beta$ is a bijection between $\mathcal O_{0,r}(m-s)$  and $\mathcal O^{\beta}_{j,r}(m):=\{\lambda \in \mathcal O_{j,r}(m)\mid \lambda^{r\mid}=\beta\}$.

Next we restrict $\zeta_\beta$ to the set $\mathcal O_{1,r}(m-s)$. Let $\mu\in \mathcal O_{1,r}(m-s)$. Suppose that the only part of $\mu^{r\mid}$ is $b_u$.  Then $\zeta_\beta(\mu)^{r\mid}=\beta\cup \mu^{r\mid}.$ We have two cases. 

\begin{enumerate}

\item[i)] If $b_u\in \beta$, then $\zeta_\beta(\mu)$ belongs to the set
$$\mathcal O^{\hat\beta}_{j,r}(m):=\{\lambda \in \mathcal O_{j,r}(m)\mid \beta \subset \lambda^{r\mid},  \lambda^{r\mid}\setminus\beta \text{ has one part size (possibly repeated)}\}.$$ 


\item[ii)] If $b_u\not\in \beta$, then $\zeta_\beta(\mu)$ belongs to the set 
$$\mathcal O^{\beta^*}_{j+1,r}(m):=\{\lambda \in \mathcal O_{j+1,r}(m)\mid \beta \subset \lambda^{r\mid},  \lambda^{r\mid}\setminus\beta \text{ has one part size (possibly repeated)}\}.$$

\end{enumerate}

Thus, $\zeta_\beta$ is a bijection between $\mathcal O_{1,r}(m-s)$ and $\mathcal O^{\hat\beta}_{j,r}(m)\cup \mathcal O^{\beta^*}_{j+1,r}(m)$.

We have 
\begin{equation}\label{bjr1}
b_{j,r}(m) = (r-1)\sum_{s\geq 0}\sum_{\beta\in \mathcal A_{j,r}(s)}\Big(|\mathcal O^{\hat\beta}_{j,r}(m)|+|\mathcal O^{\beta^*}_{j+1,r}(m) |-\ell(\beta)|\mathcal O^{\beta}_{j,r}(m)|\Big).
\end{equation}

Given a partition $\gamma\in \mathcal O_{j+1,r}(m)$, denote the different parts of $\gamma^{r\mid}$ by $(a_t)_{i_t}$, $1\leq t\leq j+1$, and let $m_t:=m_\gamma((a_t)_{i_t})$. For each $1\leq t\leq j+1$, denote by $^t\gamma$ the partition obtained from $\gamma^{r\mid}$ by removing all $m_t$ parts equal to $(a_t)_{i_t}$. Then $\gamma\in \mathcal O^{^t\gamma^*}_{j+1,r}(m)$. Therefore, $\gamma$ belongs to $j+1$ different sets 
$\mathcal O^{\beta^*}_{j+1,r}(m)$ and
\begin{equation}\label{bjr2}
\sum_{s\geq 0}\sum_{\beta\in \mathcal A_{j,r}(s)}|\mathcal O^{\beta^*}_{j+1,r}(m)|= (j+1)|\mathcal O_{j+1, r}(m)|.
\end{equation}
Each partition $\eta\in \mathcal O_{j,r}(m)$ occurs only in $\mathcal O^{\beta}_{j,r}(m)$ with $\beta=\eta^{r\mid}$. Hence, 
\begin{equation}\label{bjr3}
\sum_{s\geq 0}\sum_{\beta\in \mathcal A_{j,r}(s)}\ell(\beta)|\mathcal O^\beta_{j,r}(m)|= \sum_{\lambda\in \mathcal O_{j,r}(m)}\ell(\lambda^{r\mid})|\mathcal O^{\lambda^{r\mid}}_{j,r}(m)|.
\end{equation}

Next, assume that $\eta\in \mathcal O_{j,r}(m)$. Therefore $\eta\in \mathcal O^{\eta^{r\mid}}_{j,r}(m)$.  Denote the different parts of  $\eta^{r\mid}$ by $(a_t)_{i_t}$, $1\leq t\leq j$ and let $m_t:= m_\eta((a_t)_{i_t})$. For each $1\leq t\leq j$ and each $1\leq u \leq m_t-1$, let $^{u,t}\eta$ be the partition obtained from $\eta^{r\mid}$ by removing $u$ parts equal to $(a_t)_{i_t}$. Then $\eta\in \mathcal O^{\widehat{^{u,t}\eta}}_{j+1,r}(m)$.
It follows that $\eta\in \mathcal O^{\eta^{r\mid}}_{j,r}(m)$ belongs to $\ds\sum_{t=1}^j(m_t-1)=\ell(\eta^{r\mid})-j$ different sets $\mathcal O^{\hat\beta}_{j,r}(m)$.


Hence, together with \eqref{bjr3}, we find
\begin{align}\label{bjr4}
\sum_{s\geq 0}\sum_{\beta\in \mathcal A_{j,r}(s)}\Big(|\mathcal O^{\hat\beta}_{j,r}(m)|&-\ell(\beta)|\mathcal O^{\beta}_{j,r}(m)|\Big)\nonumber\\
& = \sum_{\lambda\in \mathcal O_{j,r}(m)}\Big((\ell(\lambda^{r\mid})-j)|\mathcal O^{\lambda^{r\mid}}_{j,r}(m)|-\ell(\lambda^{r\mid})|\mathcal O^{\lambda^{r\mid}}_{j,r}(m)|\Big)\nonumber\\ & =  -j\sum_{\lambda\in \mathcal O_{j,r}(m)}|\mathcal O^{\lambda^{r\mid}}_{j,r}(m)|\nonumber\\ 
& = -j|\mathcal O_{j,r}(m)|.
\end{align} 
Substituting \eqref{bjr2} and \eqref{bjr4} into \eqref{bjr1} completes the proof.
\end{proof}

\section{A second $n$-color Beck-type identity}

Let ${T}_{j,r}(m)$  denote the number of (distinct) parts with multiplicity greater than $r$ and less than $2r$ in all partitions in $\mathcal{D}_{j,r}(m)$.

\begin{theorem}\label{B2}
For $r\geq2$ and $m,j\geq0$, we have
\begin{equation}\label{eqnB2G}
{T}_{j+1,r}(m)-{T}_{j,r}(m)=\sum_{\lambda\in\mathcal{D}_{j,r}(m)}\bar{\ell}(\lambda)-\sum_{\lambda\in\mathcal{O}_{j,r}(m)}\bar{\ell}(\lambda),
\end{equation}
 where $\bar{\ell}(\lambda)$ is the number of distinct parts in $\lambda$.
\end{theorem}

\begin{proof}[Analytic Proof] Let $r\geq 2$. For $k\geq 0$, denote by $\widetilde{\mathcal{O}}_{j,r}(k,m)$ and $\widetilde{\mathcal{D}}_{j,r}(k,m)$ the subsets of partitions with $k$ distinct parts in $\mathcal{O}_{j,r}(m)$ and $\mathcal{D}_{j,r}(m)$, respectively.
We introduce the trivariate generating functions  
$$\widetilde{\mathcal{O}}_r(z,w,q):=
\sum_{m=0}^\infty\sum_{k=0}^\infty\sum_{j=0}^\infty|\widetilde{\mathcal{O}}_{j,r}(k,m)|z^kw^jq^m,$$ and 
$$\widetilde{\mathcal{D}}_r(z,w,q):=\sum_{m=0}^\infty\sum_{k=0}^\infty\sum_{j=0}^\infty|\widetilde{\mathcal{D}}_{j,r}(k,m)|z^kw^jq^m.$$ 
A direct expansion leads to
\begin{align*}
\widetilde{\mathcal O}_r(z,w,q)
&=\prod_{m=1}^\infty(1+wz\sum_{i=1}^{\infty}(q^{rm})^i)^m\cdot
\prod_{m=1}^\infty(1+z\sum_{i=1}^{\infty}(q^{rm})^i)^{(r-1)m}\cdot\prod_{\substack{m=1\\ r\nmid m}}^\infty(1+z\sum_{i=1}^{\infty}(q^{m})^i)^{m}   \\  
&=\prod_{m=1}^\infty\left(1+\frac{wzq^{rm}}{1-q^{rm}}\right)^m\cdot 
\prod_{m=1}^{\infty}\left(1+\frac{zq^{rm}}{1-q^{rm}}\right)^{-m}
\cdot \prod_{m=1}^{\infty}\left(1+\frac{zq^{m}}{1-q^{m}}\right)^{m}\\
&=\prod_{m=1}^{\infty}\left(\frac{1-(1-wz)q^{rm}}{1-(1-z)q^{rm}}\right)^m
\cdot \prod_{m=1}^{\infty}\left(\frac{1-(1-z)q^{m}}{1-q^{m}}\right)^{m},
\end{align*} 
and  
\begin{align*}
\widetilde{\mathcal D}_r(z,w,q)
&=\prod_{m=1}^\infty\Big(1+z\sum_{i=1}^{r-1}(q^{m})^i+wz\sum_{i=r}^{\infty}(q^{m})^i\Big)^m\\
&=\prod_{m=1}^\infty\Big(1+\frac{z(q^m-q^{rm})}{1-q^m}+\frac{wzq^{rm}}{1-q^m}\Big)^m\\
&=\prod_{m=1}^{\infty}\left(\frac{1-(1-z)q^m-(1-w)zq^{rm}}{1-q^{m}}\right)^m.
\end{align*} 
In a similar spirit, we let ${\mathcal T}_{j,r}(k,m)$ be the subset of partitions in $\mathcal{D}_{j,r}(m)$ with exactly $k$ different parts with multiplicity greater than $r$ and less than $2r$. Then
\begin{align*}
{\mathcal T}_r(z,w,q)&:=
\sum_{m=0}^\infty\sum_{k=0}^\infty\sum_{j=0}^\infty|{\mathcal T}_{j,r}(k,m)|z^kw^jq^m\\
&=\prod_{m=1}^\infty\Big(\sum_{i=0}^{r-1}(q^{m})^i+wq^{rm}+wz\sum_{i=r+1}^{2r-1}(q^{m})^i+w\sum_{i=2r}^{\infty}(q^{m})^i\Big)^m\\
&=\prod_{m=1}^\infty\Big(\frac{1-q^{rm}}{1-q^m}+wq^{rm}+\frac{wz(q^{(r+1)m}-q^{2rm})}{1-q^m}+\frac{wq^{2rm}}{1-q^m}\Big)^m\\
&=\prod_{m=1}^{\infty}\left(\frac{1-(1-w)q^{rm}-w(1-z)(q^{(r+1)m}-q^{2rm})}{1-q^{m}}\right)^m.
\end{align*} 
Setting $z=1$ in these three generating functions, it is easy to check that  
\begin{equation}\label{ODT1}
\widetilde{\mathcal O}_r(1,w,q)
=\widetilde{\mathcal D}_r(1,w,q)
={\mathcal T}_r(1,w,q)
=\prod_{m=1}^{\infty}\left(\frac{1-(1-w)q^{rm}}{1-q^{m}}\right)^m.
\end{equation}
We have 
$$\sum_{m=0}^\infty\sum_{j=0}^\infty \Big(\sum_{\lambda\in\mathcal{D}_{j,r}(m)}\bar{\ell}(\lambda)-
\sum_{\lambda\in\mathcal{O}_{j,r}(m)}\bar{\ell}(\lambda)\Big)w^jq^m 
= \left.\frac{\del}{\del z}\right|_{z=1}\!\!\!\!\!\!\bigl(\widetilde{\mathcal D}_r(z,w,q)-\widetilde{\mathcal O}_r(z,w,q)\bigr).$$
Using logarithmic differentiation, we obtain
\begin{align*}\left.\frac{\del}{\del z}\right|_{z=1}\!\!\!\!\!\!\widetilde{\mathcal O}_r(z,w,q)
&=\widetilde{\mathcal O}_r(1,w,q)\sum_{m=1}^\infty m\left(
\frac{wq^{rm}}{1-(1-w)q^{rm}}-q^{rm}+q^{m}\right)
\end{align*}
and, similarly,
\begin{align*}
\left.\frac{\del}{\del z}\right|_{z=1}\!\!\!\!\!\!\widetilde{\mathcal D}_r(z,w,q)
&= \widetilde{\mathcal D}_r(1,w,q)
\sum_{m=1}^\infty \frac{m(q^m-(1-w)q^{rm})}{1-(1-w)q^{rm}}.
\end{align*}
Subtracting these two expressions and applying identity \eqref{ODT1}, we get
\begin{align}\label{RHSB2G}
\sum_{m=0}^\infty\sum_{j=0}^\infty \Big(\sum_{\lambda\in\mathcal{D}_{j,r}(m)}\bar{\ell}(\lambda)&-
\sum_{\lambda\in\mathcal{O}_{j,r}(m)}\bar{\ell}(\lambda)\Big)w^jq^m\nonumber\\
&=\widetilde{\mathcal O}_r(1,w,q)
\sum_{m=1}^{\infty} 
\left(\frac{m(q^m-q^{rm})}{1-(1-w)q^{rm}}-m(q^m-q^{rm})\right)\nonumber\\
&=(1-w)\,\widetilde{\mathcal O}_r(1,w,q)
\sum_{m=1}^\infty \frac{m(q^{(r+1)m}-q^{2rm})}{1-(1-w)q^{rm}}.
\end{align}
On the other hand, evaluating the derivative of ${\mathcal T}_r(z,w,q)$ at $z=1$ gives
\begin{align*}
\sum_{m=0}^\infty\sum_{j=0}^\infty T_{j,r}(m)w^jq^m
&=\left.\frac{\del}{\del z}\right|_{z=1}\!\!\!\!\!\!{\mathcal T}_r(z,w,q)={\mathcal T}_r(1,w,q)
\sum_{m=1}^{\infty}\frac{mw(q^{(r+1)m}-q^{2rm})}{1-(1-w)q^{rm}},
\end{align*}
which implies 
\begin{align}\label{LHSB2G}
\sum_{m=0}^\infty\sum_{j=0}^{\infty} ({T}_{j+1,r}(m)-{T}_{j,r}(m)) w^jq^m=(1-w)\,{\mathcal T}_r(1,w,q)
\sum_{m=1}^\infty \frac{m(q^{(r+1)m}-q^{2rm})}{1-(1-w)q^{rm}}.
\end{align}
Consequently, in view of \eqref{ODT1}, \eqref{RHSB2G} coincides with \eqref{LHSB2G}, which concludes the proof.
\end{proof}

\begin{proof}[Combinatorial Proof] Let $$b'_{j,r}(m):=\sum_{\lambda\in\mathcal{D}_{j,r}(m)}\bar{\ell}(\lambda)-\sum_{\lambda\in\mathcal{O}_{j,r}(m)}\bar{\ell}(\lambda)= \sum_{\lambda\in\mathcal{D}_{j,r}(m)}\Big(\bar{\ell}(\lambda)-\bar{\ell}(\varphi_j^{-1}(\lambda))\Big).$$ For this proof, we view the parts of an $n$-color partition as elements of the set of $n$-color integers 
$$\mathbb Z_c:=\{a_i \mid a, i\in \mathbb Z^+, 1\leq i\leq a\}.$$ 
We first prove the case $j=0$ of the theorem by adapting the ideas of  \cite[Theorem 1.7]{Yang19}. Let $$\mathbb Z_c^{r \nmid}:=\{b_k\in \mathbb Z_c \mid r \nmid \gcd(b,k)\}.$$ For ease of notation, if $a_i\in \mathbb Z_c$, we write $r^t(a_i)$ for the $n$-color integer of size $r^ta$ and subscript $r^ti$. 

Given  $b_k\in \mathbb Z_c^{r \nmid}$, we define the set $$F(b_k):=\{a_i\in \mathbb Z_c \mid a_i=r^t(b_k),  t\geq 0\}.$$ 
The sets $\{F(b_k)\}_{b_k\in \mathbb Z_c^{r \nmid}}$ 
form a set partition of $\mathbb Z_c$. Hence, for any $n$-color partition $\lambda$, we have 
$$\bar\ell(\lambda)=\sum_{b_k\in \mathbb Z_c^{r \nmid}}|F(b_k)\cap \lambda|.$$ 
Let $\lambda\in \mathcal D_{0,r}(m)$. By the definition of $\varphi_0^{-1}$, it follows that $F(b_k)\cap \varphi_0^{-1}(\lambda)=\emptyset$ if and only if $F(b_k)\cap \lambda=\emptyset$. Moreover, since $\varphi_0^{-1}(\lambda)\in \mathcal O_{0,r}(m)$, we have $F(b_k)\cap \varphi_0^{-1}(\lambda)\in \{\emptyset,\{b_k\}\}$, that is, $F(b_k)\cap \varphi_0^{-1}(\lambda)$ is empty or equal to the singleton $\{b_k\}$. Hence, 
$$\bar\ell(\lambda)-\bar\ell(\varphi_0^{-1}(\lambda))=\sum_{\substack{b_k\in \mathbb Z_c^{r \nmid}\\ F(b_k)\cap \lambda\neq \emptyset}}\Big(|F(b_k)\cap \lambda|-1 \Big).$$
Let $\mathcal T_{1,r}(m):=\mathcal T_{1,r}(1,m)$, which is the subset of $n$-color partitions  of  $\lambda \in \mathcal D_{1,r}(m)$  such that $\lambda$ has a single part with multiplicity greater than $r$ and less than $2r$.  Clearly, $|\mathcal T_{1,r}(m)|=T_{1,r}(m)$.

For each $\lambda\in \mathcal D_{0,r}(m)$ and $b_k\in \mathbb Z_c^{r\nmid}$ such that $F(b_k)\cap \lambda\neq\emptyset$, we construct a subset $\mathcal T_{\lambda, r, b_k}(m)\subseteq \mathcal T_{1,r}(m)$ of size $|F(b_k)\cap \lambda|-1$ as follows. Suppose that $|F(b_k)\cap \lambda|=p$ and that the different parts in $\lambda$ that are elements of $F(b_k)$ are $r^{t_i}(b_k)$ with $0\leq t_1<t_2<\cdots < t_p$. For each integer $1\leq i\leq p$, we have $m_\lambda(r^{t_i}(b_k))\leq r-1$.
Now let $1\leq s\leq p-1$. We construct a partition $\tau_{b_k}^s$ from $\lambda$ by removing a single part equal to $r^{t_{s+1}}(b_k)$  and inserting $r$ parts equal to $r^{t_s}(b_k)$ and, for each integer $i\in [t_s+1, t_{s+1}]$, $r-1$ parts equal to $r^{i}(b_k)$.

As in \cite{Yang19}, we see that $|\tau^s_{b_k}|=|\lambda|=m$. Moreover, in the partition $\tau_{b_k}^s$, the part $r^{t_s}(b_k)$ has multiplicity greater than $r$ and less than $2r$, while all other parts have multiplicity less than $r$. Thus $\tau^s_{b_k}\in \mathcal T_{1,r}(m)$. Moreover, if $s\neq s'$, then $\tau^s_{b_k}\neq \tau^{s'}_{b_k}$. Define 
$$\mathcal T_{\lambda, r, b_k}(m):=\{\tau^s_{b_k} \mid 1\leq s \leq p-1\}.$$ 
By construction, if $b_k\neq b'_{k'}$, then $\mathcal T_{\lambda, r, b_k}(m)\cap \mathcal T_{\lambda, r, b'_{k'}}(m)=\emptyset$. 

It remains to show that $$\mathcal T_{1,r}(m)\subseteq \bigcup_{\lambda\in \mathcal D_{0,r}(m)}\Big(\bigcup_{\substack{b_k\in \mathbb Z_c^{r \nmid}\\ F(b_k)\cap \lambda\neq \emptyset}} \mathcal T_{\lambda, r, b_k}(m)\Big)$$ and  the sets $\mathcal T_{\lambda, r, b_k}(m)$ are disjoint. Let $\tau \in \mathcal T_{1,r}(m)$ and assume that the part of $\tau$ with multiplicity greater than $r$ but less than $2r$ is $r^t(b_k)$ with $b_k\in \mathbb Z_c^{r\nmid}$. From $\tau$, we create a partition $\lambda \in \mathcal D_{0,r}(m)$ using the following algorithm.

\begin{itemize}

\item[(1)] Replace $r$ parts equal to $r^t(b_k)$ with a single part equal to $r^{t+1}(b_k)$. Let $\lambda$ be the resulting partition, and set the counter $i=1$.

\item[(2)] While $\lambda$ contains $r$ parts equal to $r^{t+i}(b_k)$, we modify $\lambda$ by replacing $r$ parts equal to $r^{t+i}(b_k)$ with a single part equal to $r^{t+i+1}(b_k)$ and then increase $i$ by 1.
Otherwise, we stop.
\end{itemize}

Since the partition $\lambda$ has a finite number of parts, step (2) will run a finite number of times and eventually produce the partition $\lambda \in \mathcal{D}_{0,r}(m)$. 
Moreover, the set $F(b_k)\cap \lambda$ is non-empty  because it contains $r^t(b_k)$.
We have $\tau \in \mathcal T_{\lambda, r, b_k}(m)$. 
 The algorithm above shows that $\tau\in \mathcal T_{1,r}(m)$ uniquely determines $\lambda$ and $b_k$.  Thus, if $(\lambda, b_k)\neq (\lambda', b'_{k'})$, the set $\mathcal{T}_{\lambda, r, b_k}(m)\cap \mathcal{T}_{\lambda', r, b'_{k'}}(m)\subseteq \mathcal T_{1,r}(m)$ must be empty. Hence, the sets $\mathcal T_{\lambda, r, b_k}(m)$ are disjoint.  
This completes the combinatorial proof for $b'_{0,r}(m)=T_{1,r}(m)$.

Next, let $j\geq 1$. The combinatorial proof of this case is similar to that of \cite[Theorem 6]{BW21}. We adopt the notation from the combinatorial proof of Theorem \ref{F} and write $\lambda\in \mathcal D_{j,r}(m)$ as $\lambda= \lambda^{<r}\cup \lambda^r$. We have 
$$b'_{j,r}(m)=\sum_{\lambda\in \mathcal D_{j,r}(m)}\Big(\bar\ell(\lambda)-\bar\ell(\varphi_j^{-1}(\lambda))\Big).$$ 
Let $\lambda \in \mathcal{D}_{j,r}(m)$ and denote by $\bar g(\lambda)$ the number of different parts $a_i$ in $\lambda$ with $m_\lambda(a_i)\geq r$ and $r\nmid m_\lambda(a_i)$. It follows that 
$$\bar\ell(\lambda)=\bar\ell(\lambda^{<r})+\bar\ell(\lambda^r)-\bar g(\lambda).$$ 
On the other hand, $\varphi_j^{-1}(\lambda)=\varphi_0^{-1}(\lambda^{<r})\cup \psi^{-1}(\lambda^r)$ and the union is disjoint. Moreover, $\bar\ell(\psi^{-1}(\lambda^r))=\bar\ell(\lambda^r)$. Therefore, 
$$\bar\ell(\varphi_j^{-1}(\lambda))=\bar\ell(\varphi_0^{-1}(\lambda^{<r}))+\bar\ell(\lambda^{r}).$$
We denote by $\mathcal B_{j,r}(s)$ the set of partitions in $\mathcal D_{j,r}(s)$ with $j$ different parts, all occurring with multiplicity divisible by $r$. The set $\mathcal B_{j,r}(s)$ may be empty for certain values of $s$. We have 
\begin{align}
b'_{j,r}(m)& \notag =\sum_{\lambda\in \mathcal D_{j,r}(m)}
\left(\bar\ell(\lambda^{<r})-\bar\ell(\varphi_0^{-1}(\lambda^{<r}))-\bar g(\lambda)\right)\\ 
& \notag =\sum_{s\geq 0}\sum_{\beta\in \mathcal B_{j,r}(s)} \sum_{\alpha\in \mathcal D_{0,r}(m-s)}\left( \bar\ell(\alpha)-\bar\ell(\varphi_0^{-1}(\alpha))\right)-\sum_{\lambda\in \mathcal D_{j,r}(m)} \bar g(\lambda)\\
\notag & =\sum_{s\geq 0}\sum_{\beta\in \mathcal B_{j,r}(s)}b'_{0,r}(m-s)-\sum_{\lambda\in \mathcal D_{j,r}(m)} \bar g(\lambda)\\ 
\label{b'} & = \sum_{s\geq 0}\sum_{\beta\in \mathcal B_{j,r}(s)}T_{1,r}(m-s)-\sum_{\lambda\in \mathcal D_{j,r}(m)} \bar g(\lambda).
\end{align}

The double sum in \eqref{b'} is equal to the  number of pairs $(\mu, \beta)$ with $\beta\in \mathcal B_{j,r}(s)$ for some $s$ and $\mu \in \mathcal T_{1,r}(m-s)$. Given such a pair $(\mu, \beta)$, we consider the $n$-color partition $\eta:=\mu\cup \beta$. Suppose that the only part of $\mu$ with multiplicity at least $r$ is $b_u$. Then $r<m_\mu(b_u)<2r$. We have two cases. 

\begin{itemize}
\item[(1)] If $b_u\in \beta$, then $\eta\in \mathcal D_{j,r}(m)$, $\beta\subset \eta^r$, $\eta^r\setminus \beta$ has exactly $r$ parts equal to $b_u$, $m_\eta(b_u)>2r$ and $r\nmid m_\eta(b_u)$.
Thus, $\eta$ belongs to the set  
$$\mathcal D^{\hat\beta}_{j,r}(m):=\{\lambda \in \mathcal D_{j,r}(m)\mid \beta \subset \lambda^r, |\lambda^r\setminus \beta|=r, x\in \lambda^r\setminus \beta \implies m_\lambda(x)\geq 2r \text{ and } r\nmid m_\lambda(x)\}.$$

\item[(2)] If $b_u\not\in \beta$, then $\eta\in \mathcal D_{j+1,r}(m)$, $\beta\subset \eta^r$, $\eta^r\setminus \beta$ has exactly $r$ parts equal to $b_u$  and $r<m_\eta(b_u)<2r$. Thus, $\eta$ belongs to the set
$$\mathcal D^{\beta^*}_{j+1,r}(m):=\{\lambda \in \mathcal D_{j+1,r}(m) \mid \beta \subset \lambda^r, x\in \lambda^r\setminus \beta \implies r<m_\lambda(x)<2r \}.$$

\end{itemize}

Hence, the double sum in \eqref{b'} is equal to
$$\sum_{s\geq 0}\sum_{\beta\in \mathcal B_{j,r}(s)}\Big(|\mathcal D^{\hat\beta}_{j,r}(m)|+|\mathcal D^{\beta^*}_{j+1,r}(m)|\Big).$$
Given a partition $\xi\in \mathcal D_{j,r}(m)$, for each part $a_i$ in $\xi$ with $m_\xi(a_i)\geq 2r$ and $r\nmid m_\xi(a_i)$, let $\beta_{a_i}$ be the partition obtained from $\xi^r$ by removing $r$ parts equal to $a_i$. Then $\xi\in \mathcal D^{\widehat{\beta_{a_i}}}_{j,r}(m)$. Thus $\sum_{s\geq 0}\sum_{\beta\in \mathcal B_{j,r}(s)}|\mathcal D^{\hat\beta}_{j,r}(m)|$ equals the number  of all (different) parts with multiplicity at least $2r$ and not divisible by $r$ in all partitions in $\mathcal D_{j,r}(m)$. 

By a similar argument, $\sum_{s\geq 0}\sum_{\beta\in \mathcal B_{j,r}(s)}|\mathcal D^{\beta^*}_{j+1,r}(m)|$ equals the number  of all (different) parts with multiplicity greater than $r$ and less than $2r$ in all partitions in $\mathcal D_{j+1,r}(m)$, i.e., $T_{j+1,r}(m)$.

It follows that
\begin{equation}\label{tt}
\sum_{s\geq 0}\sum_{\beta\in \mathcal B_{j,r}(s)}T_{1,r}(m-s)=\sum_{s\geq 0}\sum_{\beta\in \mathcal B_{j,r}(s)}|\mathcal D^{\hat\beta}_{j,r}(m)|+T_{j+1,r}(m).
\end{equation}
On the other hand, by the definition of $\bar g(\lambda)$, 
\begin{equation}\label{gg}
\sum_{\lambda\in \mathcal D_{j,r}(m)} \bar g(\lambda)=\sum_{s\geq 0}\sum_{\beta\in \mathcal B_{j,r}(s)}|\mathcal D^{\hat\beta}_{j,r}(m)|+T_{j,r}(m).
\end{equation}
Finally, substituting \eqref{tt} and \eqref{gg} into \eqref{b'} completes the combinatorial proof.
\end{proof}

\end{document}